\pdfoutput=1
\documentclass[11pt]{amsart}
\usepackage{pdfsync}
\usepackage{enumerate}
\usepackage{graphicx}
\usepackage{latexsym}
\usepackage{amsfonts}
\usepackage[utf8]{inputenc}
\usepackage{amsthm}
\usepackage{amsmath}
\usepackage{amssymb}
\usepackage[
            unicode,
            breaklinks=true,
            colorlinks=true]{hyperref}
\usepackage{amscd}
\usepackage[all]{xy} 
\usepackage{mathrsfs} 
\usepackage{stmaryrd} 
\usepackage{amsopn} 
\usepackage{eepic}
\usepackage{epic}
\usepackage{eucal}

\usepackage{epsfig}
\usepackage{psfrag}

\newtheorem{theorem}{Theorem}[section]
\newtheorem{prop}[theorem]{Proposition}

\newtheorem{lemma}[theorem]{Lemma}

\newtheorem{cor}[theorem]{Corollary}

\newcommand{\R}{{\mathbb R}}
\newcommand{\Z}{{\mathbb Z}}

\newcommand{\N}{{\mathbb N}}
\newcommand{\T}{{\mathbb T}}

\newcommand{\ham}[1]{\mathcal{X}_{#1}}

\newcommand{\op}[1]{\!\!\mathop{\rm ~#1}\nolimits}

\renewcommand{\geq}{\geqslant}
\renewcommand{\leq}{\leqslant}
\newcommand{\pscal}[2]{\langle #1,#2\rangle}

\newcommand{\abs}[1]{\left|#1\right|}

\newcommand{\h}{\hbar}

\newenvironment{remark}{\refstepcounter{theorem}\par\medskip\noindent{\bf
Remark~\thetheorem.}}{\unskip\nobreak\hfill\hbox{}}

\newenvironment{definition}{\refstepcounter{theorem}\par\medskip\noindent{\bf
Definition~\thetheorem.}}{\unskip\nobreak\hfill\hbox{}}

\newtheorem{theoremL}{Theorem}

\begin{document}

\title[Inverse spectral theory for Jaynes\--Cummings systems]{Inverse
  spectral theory for semiclassical Jaynes\--Cummings systems}
\author{Yohann Le Floch 
\,\,\,\, \,\,\,\,\,\,\,\,\,\,\,\,\,\,\,\,\,\,\,\,\,
\'Alvaro Pelayo
\,\,\,\, \,\,\,\,\,\,\,\,\,\,\,\,\,\,\,\,\,\,\,\,\,
San V\~u Ng\d
  oc} \date{}

\maketitle
\thispagestyle{empty}
 
\begin{abstract}
  Quantum semitoric systems form a large class of quantum Hamiltonian
  integrable systems with circular symmetry which has received great
  attention in the past decade. They include systems of high interest
  to physicists and mathematicians such as the Jaynes\--Cummings model
  (1963), which describes a two-level atom interacting with a
  quantized mode of an optical cavity, and more generally the
  so-called systems of Jaynes\--Cummings type. In this paper we
  consider the joint spectrum of a pair of commuting semiclassical
  operators forming a quantum integrable system of Jaynes\--Cummings
  type.  We prove, assuming the Bohr\--Sommerfeld rules hold, that if
  the joint spectrum of two of these systems coincide up to
  $\mathcal{O}(\hbar^2)$, then the systems are isomorphic.
\end{abstract}

\section{Introduction}
\label{sec:introduction}

A natural question in semiclassical analysis is whether the knowledge
of the joint spectrum of a quantum integrable system allows to
determine the classical dynamics of the underlying integrable
system. Pursuing this question in such generality has been made
possible thanks to the development of semiclassical analysis with
microlocal techniques (see for instance the recent books by
Dimassi\--Sj\"ostrand \cite{DiSj99},
Guillemin\--Sternberg~\cite{GuSt2012}, and Zworski~\cite{Zw2012} and
the references therein) which nowadays permits a constant interaction
between symplectic geometry and spectral theory.  In particular, these
techniques led to the resolution of the inverse spectral question in a
number of cases; for instance: (i) compact toric integrable systems,
in the context of Berezin\--Toeplitz quantization~\cite{ChPeVN2013};
(ii) semiglobal inverse problem near the so called ``focus\--focus"
singularities of 2D integrable systems, in the context of
$\hbar$\--pseudodifferential quantization~ \cite{PeVN2013}; (iii)
inverse theory for the Laplacian on surfaces of
revolution~\cite{Ze98}; (iv) $1$\--dimensional pseudodifferential
operators with Morse symbol~\cite{VN2011}. The flexibility of
microlocal analysis makes us hope that more general integrable systems
will be treated in the future.  An interesting step is to understand
what happens for \emph{semitoric} systems on $4$-dimensional phase
spaces \cite{PeVN2011}, which form an important extension of toric
systems.

\begin{definition}
  \label{def:semitoric}
  A $\mathcal{C}^\infty$ \emph{classical integrable system} $F:=(J,H)
  \colon M \to \R^2$ on a connnected symplectic $4$\--dimensional
  manifold $(M,\omega)$ is \emph{semitoric} if:
  \begin{enumerate}[({\rm H}.i)]
  \item\label{Hii} $J$ is the momentum map of an effective Hamiltonian
    circle action.
  \item \label{Hiii} The singularities of $F$ are non-degenerate with
    no hyperbolic component.
  \item \label{H-proper} $J$ is a proper map (i.e., the preimages of
    compact sets are compact).
  \end{enumerate}
  A \emph{quantum semitoric integrable system} $(P,Q)$ is given by two
  semiclassical commuting self\--adjoint operators whose principal
  symbols form a classical semitoric integrable system. The notion of
  semiclassical operators that we use is defined in
  Section~\ref{sec:semicl-oper}; it includes standard semiclassical
  pseudodifferential operators, and Berezin-Toeplitz operators.
\end{definition}

When $(M,\omega)$ is four-dimensional, toric systems form a particular
class of semitoric systems for which $F$ is the momentum map of a
Hamiltonian $\mathbb{T}^2$-action. The symplectic classification of
toric systems was done by Delzant~\cite{Del88}, and the quantum
spectral theory in the case of Berezin-Toeplitz quantization was
carried out in~\cite{ChPeVN2013}.  A simple corollary of this spectral
theory is that the image $F(M)$, which is the so-called Delzant
polytope, can be recovered from the joint spectrum; in view of the
Delzant theorem, this implies that the joint spectrum completely
determines the triple $(M,\omega,F)$ up to toric isomorphism.

The main difference with the toric case is that focus-focus
singularities can appear in a semitoric system, making the system more
difficult to describe. For instance, if there is at least one
focus-focus singularity, the image of the moment map is no longer a
convex polygon. Moreover, new symplectic invariants appear; according
to \cite{PeVN2009,PeVN2011}, a semitoric system is determined up to
isomorphisms\footnote{The notion of isomorphism for semitoric systems
  is recalled in Definition~\ref{def:iso}.} by five symplectic
invariants:
\begin{enumerate}
\item the \emph{number} of focus-focus singular values of the system;
\item a \emph{Taylor series} $\sum_{i,j \in \N} a_{ij} X^i Y^j$ for
  each focus\--focus singularity (\cite{VN2003, PeVN2009});
\item a \emph{height invariant} $h>0$ measuring the volume of certain
  reduced spaces at each focus\--focus singularity;
\item a \emph{polygonal invariant} (in fact, a family of polygons)
  obtained by unwinding the singular affine structure of the system;
\item an index associated with each pair of focus-focus singularities,
  called the \emph{twisting index}.
\end{enumerate}
Therefore, if one is able to recover these five invariants from the
semiclassical joint spectrum of a quantum integrable system quantizing
$(J,H)$, then in effect one can recover the triple $(M,\omega,F)$ up
to the appropriate notion of isomorphism. From \cite{PeVN2013}, (2)
can be recovered. The goal of this paper is to show that this is the
case for invariants (1) to (4).  Let us be more precise and state our
main result. We will say that a semitoric integrable system $F =
(J,H)$ is \emph{simple} if it satisfies the following: if $m$ is a
focus-focus critical point for $F$, then $m$ is the unique critical
point of the level set $J^{-1}(J(m))$. Our main theorem is the
following.

\begin{theoremL} \label{thm:main} Let $(P,Q)$ be a quantum simple
  semitoric system on $M$ for which the Bohr\--Sommerfeld rules
  hold. Then from the knowledge of the semiclassical joint spectrum
  ${\rm JointSpec}(P,Q) + \mathcal{O}(\hbar^2)$, one can recover the
  four following invariants of the associated classical semitoric
  system:
  \begin{enumerate}
  \item[{\rm (1)}] the number $m_{f}$ of focus-focus values,
  \item[{\rm (2)}] the Taylor series associated with each focus-focus
    value,
  \item[{\rm (3)}] the height invariant associated with each
    focus-focus value,
  \item[{\rm (4)}] the polygonal invariant of the system.
  \end{enumerate}
\end{theoremL}

Of course, Theorem~\ref{thm:main} is not entirely satisfactory if one
has in mind the problem of completely recovering the classical system
from the joint spectrum of its quantum counterpart. However, there is
one case where we can say more: for the simplest examples of semitoric
integrable systems, which we call systems \emph{of Jaynes\--Cummings
  type}. The characteristic of such a system is to display only one
focus-focus singularity.  One of the simplest yet most important
models in classical and quantum mechanics was proposed by Jaynes and
Cummings \cite{JaCu1963, Cu1965} in 1963, and it is now known as the
\emph{Jaynes\--Cummings model}.\footnote{The Jaynes-Cummings model was
  initially introduced to describe the interaction between an atom
  prepared in a mixed state with a quantum particle in an optical
  cavity. It was found to apply to many physical situations (quantum
  chemistry, quantum optics, quantum information theory, etc.) because
  it represents the easiest way to have a finite dimensional state
  (like a spin) interact with an oscillator.
}The Jaynes\--Cummings model is obtained by coupling a spin with a
harmonic oscillator. In this way one obtains a physical system with
phase space $S^2 \times \R^2$ and Hamiltonian functions $
J:=\frac{u^2+v^2}{2}+z$, $H=\frac{ux+vy}{2}$, where $(x,y,z)$ denotes
the point in the $2$\--sphere $S^2\subset\R^3$ and $(u,v)$ denote
points in $\R^2$ ($J$ is the momentum map for the combined rotational
$S^1$\--actions about the origin in $\R^2$ and about the vertical axes
on $S^2$).  Recently the second and third authors described in full
the semiclassical spectral theory of this system~\cite{PeVN2012}.

\begin{definition}
  A classical integrable system $F:=(J,H) \colon M \to \R^2$ on a symplectic $4$\--manifold
  $(M,\omega)$ is of \emph{Jaynes\--Cummings type} if:
  \begin{itemize}
  \item[(a)] $F$ is a semitoric system;
  \item[(b)] $F$ has one, and only one, singularity of focus\--focus
    type.
  \end{itemize}
  A \emph{quantum integrable system $(P,Q)$ of Jaynes\--Cummings type}
  is given by two semiclassical commuting self\--adjoint operators
  whose principal symbols form a classical integrable system of
  Jaynes\--Cummings type.
\end{definition}

The Jaynes\--Cummings model is a particular example of a system of
Jaynes\--Cummings type.  Jaynes\--Cummings type systems form a large
class of integrable Hamiltonian systems because the structure of the
singularity in part (b) is extremely rich, and it is classified by a
Taylor series $\sum_{i,j \in \N} a_{ij} X^i Y^j$, according to
\cite{VN2003} (two such singularities are symplectically equivalent if
and only if each and everyone of the coefficients in the Taylor series
coincide for both singularities). Moreover, by \cite{VN2003, PeVN2011}
every such Taylor series can be realized as the Taylor series
invariant of an integrable system (in fact, of \emph{many}
inequivalent such systems). Accordingly, the moduli space of
Jaynes\--Cummings type systems is, from the point of view of
Hamiltonian dynamics, extremely rich.  As a corollary of Theorem
\ref{thm:main}, we solve the inverse spectral problem for quantum
integrable systems of Jaynes\--Cummings type.
\begin{theoremL} \label{main2} Let $(P,Q)$ be a quantum integrable
  system of Jaynes-Cummings type on $M$ for which the
  Bohr\--Sommerfeld rules hold.  Then from the knowledge of ${\rm
    JointSpec}(P,Q) + \mathcal{O}(\hbar^2)$, one can recover the
  principal symbol $\sigma(P,Q)$ up to isomorphisms of semitoric
  integrable systems.
\end{theoremL}

This theorem gives the first global inverse spectral result that the
authors are aware of for integrable Hamiltonian systems with
focus-focus singularities (and hence no global action-angle
variables).  In the context of Hamiltonian toral actions (eg. toric
integrable systems), all singularities are of elliptic type, which is
strongly related to the dynamical and spectral rigidity of such
systems~\cite{ChPeVN2013}. We believe that allowing focus-focus
singularities, which have a much larger moduli space, is an important
step forward in the study of the inverse spectral problem for general
integrable systems.

The problem treated in this paper belongs to a class of semiclassical
inverse spectral questions which has attracted much attention in
recent years, e.g.  \cite{LF14, Gu95, Ha13, CoGu2011, PPV14, phan14,
  VN2011}, which goes back to pioneer works of B\'erard~\cite{Be76},
Br\"uning\--Heintze~\cite{BH84}, Colin de Verdi{\`e}re~\cite{CdV,
  CdV2}, Duistermaat\--Guillemin~\cite{DuGu1975}, and
Guillemin\--Sternberg \cite{GuSt}, in the 1970s/1980s, and are closely
related to inverse problems that are not directly semiclassical but do
use similar microlocal techniques for some integrable systems, as
in~\cite{Ze98} (see also~\cite{Ze04} and references therein).

\vskip 1em

We conclude this section by a natural question.  The following
corollary of Theorem \ref{thm:main} directly follows from the
symplectic classification~\cite{PeVN2009} of semitoric systems:
\begin{cor}
  Let $(P,Q)$ and $(P',Q')$ be quantum simple semitoric systems on $M$ and
  $M'$, respectively, for which the Bohr\--Sommerfeld rules hold.  If
  \begin{eqnarray} {\rm JointSpec}(P,Q)={\rm JointSpec}(P',Q') +
    \mathcal{O}(\hbar^2), \label{x}
  \end{eqnarray}
  and if the twisting index invariants of $\sigma(P,Q)$ and
  $\sigma(P',Q')$ are equal, then $\sigma(P,Q)$ and $\sigma(P',Q')$
  are isomorphic as semitoric integrable systems.
\end{cor}

In view of this result, one question remains: can one obtain the
twisting index invariant of a semitoric system from the data of the
joint spectrum of the corresponding quantum system? A positive answer
to this question would lead to the definition of a new quantum
invariant which would be quite robust (since the twisting index
between two focus-focus singularities is just an integer).

\section{Preliminaries} \label{ba}

Let $(M,\omega)$ be a smooth, connected $4$\--dimensional symplectic
manifold.

\subsection{Integrable systems}
An \emph{integrable system} $(J,H)$ on $(M,\omega)$ consists of two
Poisson commuting functions $J,H \in
\mathcal{C}^{\infty}(M;\mathbb{R})$ i.e.~:
 $$
 \{J,H\}:=\omega(\ham{J},\, \ham{H})=0,
 $$
 whose differentials are almost everywhere linearly independent
 $1$\--forms.  Here $\ham{J},\ham{H}$ are the Hamiltonian vector
 fields induced by $J,H$, respectively, via the symplectic form
 $\omega$: $ \omega(\ham{J},\cdot)=-{\rm d}J$,
 $\omega(\ham{H},\cdot)=-{\rm d}H$. Moreover, the function $F=(J,H)$
 will be assumed to be proper throughout this paper.
  
 For instance, let $M_0={\rm T}^*\T^2$ be the cotangent bundle of the
 torus $\T^2$, equipped with canonical coordinates
 $(x_1,x_2,\xi_1,\xi_2)$, where $x\in \T^2$ and $\xi\in {\rm
   T}^*_x\T^2$. The linear system $$(J_0,H_0):=(\xi_1, \xi_2)$$ is
 integrable.

 An \emph{isomorphism} of integrable systems $(J,\,H)$ on $(M,\omega)$
 and $(J',\,H')$ on $(M',\omega')$ is a diffeomorphism $ \varphi
 \colon M \to M'$ such that $\varphi^*\omega'=\omega$ and
$$
\varphi^*(J',\,H')=(f_1(J,\,H),\,f_2(J,\,H))
$$
for some local diffeomorphism $(f_1,f_2)$ of $\R^2$. This same
definition of isomorphism extends to any open subsets $U\subset M$,
$U'\subset M'$ (and this is the form in which we will use it
later). Such an isomorphism will be called \emph{semiglobal} if $U,U'$
are respectively saturated by level sets $\{J=\text{const}_1,
H=\text{const}_2\}$ and $\{J'=\text{const}'_1, H'=\text{const}'_2\}$.

If $F=(J,H)$ is an integrable system on $(M,\omega)$, consider a point
$c\in\R^2$ that is a \emph{regular value} of $F$, and such that the
fiber $\Lambda_{c} = F^{{-1}}(c)$ is compact and connected. Then, by
the action-angle theorem~\cite{duistermaat}, a saturated neighborhood
of the fiber is \emph{isomorphic} in the previous sense to the above
linear model on $M_0={\rm T}^*\T^2$. Therefore, all such regular
fibers (called \emph{Liouville tori}) are isomorphic in a
neighborhood.

However, the situation changes drastically when the condition that $c$
be regular is violated. For instance, it has been proved
in~\cite{VN2003} that, when $c$ is a so-called \emph{focus-focus}
critical value, an infinite number of equations has to be satisfied in
order for two systems to be semiglobally isomorphic near the critical
fiber.

\subsection{Semitoric systems}

Semitoric systems (Definition~\ref{def:semitoric}) form a particular
class of integrable systems admitting an $S^1$ symmetry. It is
therefore natural to introduce a suitable notion of isomorphism for
such systems, which mixes the general notion defined in the previous
section with the more rigid one coming from Hamiltonian
$S^1$-manifolds.

\begin{definition} \label{def:iso} The semitoric systems
  $(M_1,\, \omega_1,\,F_1:=(J_1,\,H_1))$ and \\ $(M_2,\, \omega_2,\,
  F_2:=(J_2,\,H_2))$ are \emph{isomorphic} if there exists a
  symplectomorphism $\varphi: M_1 \to M_2$ such that
  $\varphi^*(J_2,\,H_2)=(J_1,\,h(J_1,\,H_1))$ for a smooth $h$ such
  that $\frac{\partial h}{\partial H_1}>0$.
\end{definition}

\subsection{The period lattice}
\label{subsect:period}

Let $F=(J,H)$ be an integrable system on a $4$-dimensional symplectic
manifold. For any regular value $c$ of $F$, the set of points $(t,u)
\in \R^2$ such that the vector field $t\ham{J} + u\ham{H}$ has a
$2\pi$-periodic flow on $\Lambda_{c}$ is a sublattice of $\R^2$ called
the \emph{period lattice} \cite{duistermaat}. When $c$ varies in the
set of regular values of $F$, the collection of the period lattices is
a Lagrangian subbundle of $T^*\R^2$, called the period bundle.

Coming back to our case where $F=(J,H)$ is a semitoric system, there
is a natural way to construct a basis of this lattice. Firstly, since
$J$ generates a $S^1$-action, $(1,0)$ belongs to the period
lattice. Secondly, define two real numbers $\tau_{1}(c), \tau_{2}(c)$
as follows: choose a point $m \in \Lambda_{c}$, and define
$\tau_{2}(c) > 0$ as the time of first return for the Hamiltonian flow
associated with $H$ to the trajectory of the Hamiltonian flow of $J$
passing through $m$. Let $\tau_{1}(c) \in [0,2\pi)$ be the time that
it takes to come back to $m$ following the flow of $\ham{J}$. Because
of the commutativity of the Hamiltonian flows of $J$ and $H$, the
values of $\tau_{1}(c), \tau_{2}(c)$ do not depend on the choice of
the starting point $m \in \Lambda_{c}$. The vector field
$\tau_{1}(c)\ham{J} + \tau_{2}(c)\ham{H}$ defines a $1$-periodic flow;
hence, if we define
\begin{equation}
  \zeta_{1}(c) = \frac{\tau_{1}(c)}{2\pi} , \quad \zeta_{2}(c) = \frac{\tau_{2}(c)}{2\pi},
  \label{equ:zeta} 
\end{equation}
then $(\zeta_{1}(c), \zeta_{2}(c))$ and $(0, 1)$ form a basis of the
period lattice.

\subsection{Semiclassical operators}
\label{sec:semicl-oper}
Let $I \subset (0,1]$ be any set which accumulates at $0$.  If
$\mathcal{H}$ is a complex Hilbert space, we denote by
$\mathcal{L}(\mathcal{H})$ the set of linear (possibly unbounded)
self-adjoint operators on $\mathcal{H}$ with a dense domain.

A space $\Psi$ of \emph{semiclassical operators} is a subspace of
$\prod_{\hbar \in I} \mathcal{L}(\mathcal{H}_{\hbar})$, containing the
identity, and equipped with a weakly positive principal symbol map,
which is an $\R$-linear map $$\sigma \colon \Psi \to
\mathcal{C}^{\infty}(M;\, \R),$$ with the following properties:
\begin{enumerate}
\item \label{item:normalization} $\sigma(I)=1$; (\emph{normalization})
\item \label{item:product} if $P, Q$ are in $\Psi$ and if the
  composition $P\circ Q$ is well defined and is in $\Psi$, then
  $\sigma(P\circ Q) = \sigma(P)\sigma(Q)$; (\emph{product formula})
\item \label{item:garding }if $\sigma(P)\geq 0$, then there exists a
  function $\h\mapsto\epsilon(\h)$ tending to zero as $\h\to 0$, such
  that $P\geq -\epsilon(\h)$, for all $\h\in I$. (\emph{weak
    positivity})
\end{enumerate}

If $P=(P_{\hbar})_{\hbar \in I} \in \Psi$, the image $\sigma(P)$ is
called the \emph{principal symbol of $P$}.

There are two major examples of such semiclassical operators. One is
given by semiclassical pseudodifferential operators, as described for
instance in~\cite{DiSj99} or~\cite{Zw2012}. The second one is less
well known, but developing very fast: semiclassical (or Berezin) -
Toeplitz operators, as described in~\cite{BorPauUri,Ch2003b, Ch2006b,
  Ch2006a, Ch2007,MaMar2008,Schli2010} following the pioneer
work~\cite{BG81}.

\subsection{Semiclassical spectrum}

Recall that when $A$ and $B$ are unbounded self-adjoint operators, they
are said to commute when their projector-valued spectral measures
commute.

If $P=(P_{\hbar})_{\hbar \in I}$ and $Q=(Q_{\hbar})_{\hbar \in I}$ are
semiclassical operators on $(\mathcal{H}_\h)_{\hbar \in I}$, in the
sense of Section~\ref{sec:semicl-oper}, we say that they commute if
for each $\hbar \in I$ the operators $P_{\hbar}$ and $Q_{\hbar}$
commute.

 \begin{figure}[h]
   \centering
   \includegraphics[width=0.35\linewidth]{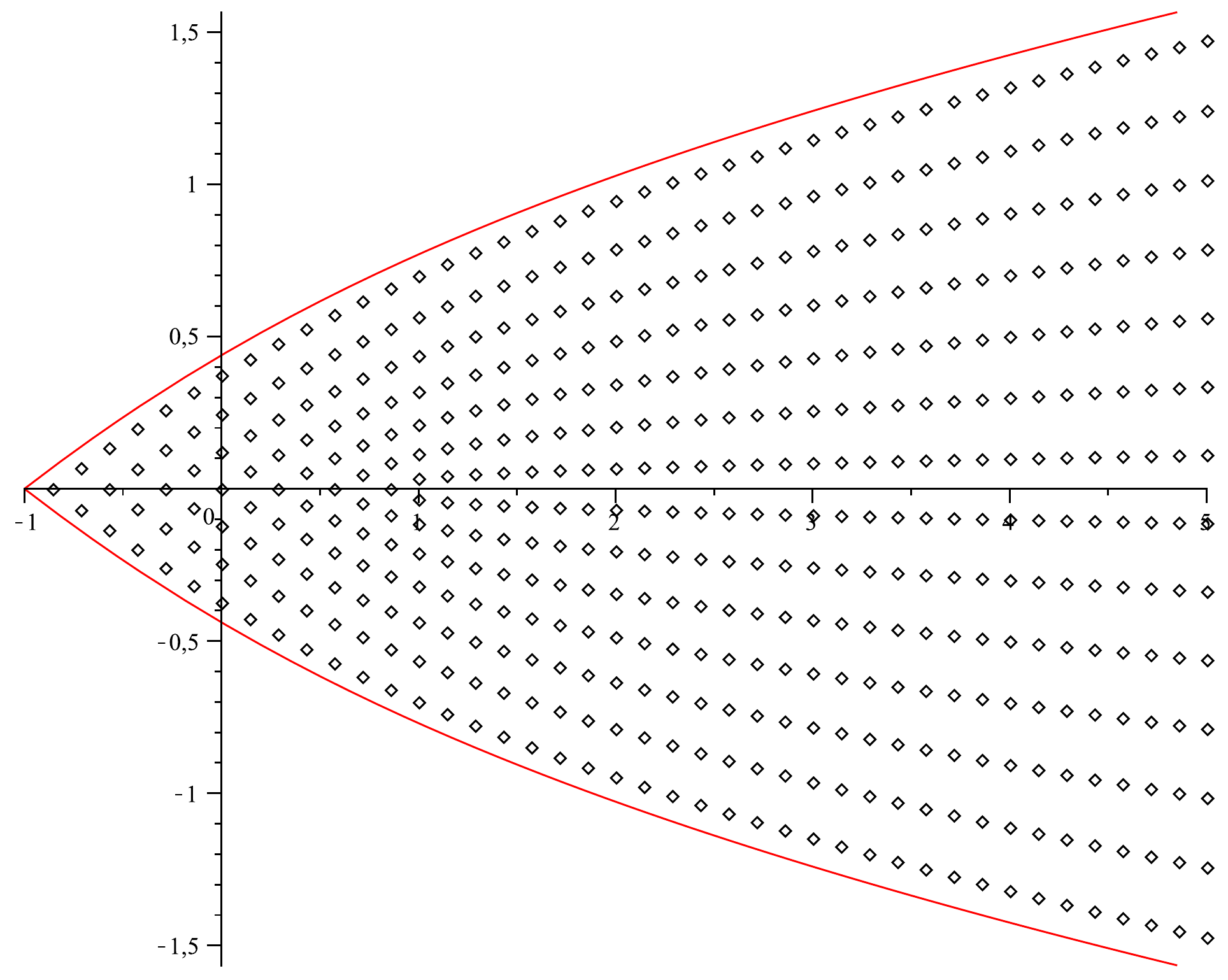}
   \caption{Semiclassical joint spectrum of the Jaynes\--Cummings model.}
   \label{fig:spectrumapprox2}
 \end{figure}

 If $P$ and $Q$ commute, we may define for fixed $\hbar$, the
 \emph{joint spectrum} of $(P_{\hbar},Q_{\hbar})$ to be the support of
 the joint spectral measure. It is denoted by
 $\op{JointSpec}(P_{\hbar},\,Q_{\hbar})$. If $\mathcal{H}_\h$ is
 finite dimensional (or, more generally, when the joint spectrum is
 discrete), then
 $$
 \op{JointSpec}(P_{\hbar},\,Q_{\hbar})=\Big\{(\lambda_1,\lambda_2)\in
 \R^2\,\, |\,\, \exists v\neq 0,\,\, P_{\hbar} v = \lambda_1
 v,\,\,Q_{\hbar} v = \lambda_2 v \Big\}.
 $$
 The \emph{joint spectrum} of $P,Q$ is the collection of all joint
 spectra of $(P_{\hbar},Q_{\hbar})$, $\hbar \in I$. It is denoted by
 $\op{JointSpec}(P,\,Q)$. For convenience of the notation, we will
 also view the joint spectrum of $P,Q$ as a set depending on $\hbar$.

 \subsection{Joint spectrum and image of the joint principal symbol}

\begin{prop}\label{prop:joint-spectrum-image}
  If $F:=(J,H) \colon M \to \mathbb{R}^2$ is an integrable system on a
  $4$\--dimensional connected symplectic manifold and $P, Q$ are
  commuting semiclassical operators with principal symbols $J,H \colon
  M \to \mathbb{R}$, then
  \begin{align*}
    E \notin F(M) \Rightarrow \quad & \exists \varepsilon > 0 \quad
    \exists \hbar_0 \in I \quad \forall \hbar \leq \hbar_0 \in
    I,\nonumber \\ & \mathrm{JointSpec}(P_{\hbar},Q_{\hbar}) \cap
    B(E,\varepsilon) = \emptyset.
  \end{align*}
\end{prop}
This proposition is well-known for pseudodifferential and Toeplitz
operators; it is interesting to notice that, in fact, it directly
follows from the axioms we chose for semiclassical operators in
Section~\ref{sec:semicl-oper}.
\begin{proof}
  If $E = (E_{1}, E_{2})$ does not belong to $F(M)$, then the function
  \[ f = (J-E_{1})^2 + (H-E_{2})^2 \] never vanishes. Thus, by the
  normalization, the product rule and the weak positivity of the
  principal symbol (items~\ref{item:normalization},~\ref{item:product}
  and ~\ref{item:garding } in section~\ref{sec:semicl-oper}), we have
  \begin{equation}
    (P-E_1)^2 + (Q-E_2)^2 \geq C,
    \label{eq:positive}
  \end{equation}
  for some constant $C>0$, when $\h$ is small enough. If fact, since
  $F(M)$ is closed (because $F$ is proper), the same holds uniformly
  when $E$ varies is a small ball. Let $\Pi_Q(d\lambda)$ and
  $\Pi_P(d\mu)$ be the spectral measures of $P$ and $Q$ respectively
  (now $\h$ is fixed). We have
  \[
  (P-E_1)^2 + (Q-E_2)^2 = \int (\lambda-E_1)^2 \Pi_P({\rm d}\lambda) +
  \int (\mu-E_2)^2 \Pi_Q({\rm d}\mu).
  \]
  Suppose that $(E_1,E_2)$ belongs to the joint spectrum of
  $(P,Q)$. Then for each $n\geq 0$ one can find a vector $u_n$ of norm
  1 such that
  \[
  u_n\in \textup{Ran}(\Pi_P([E_1-\tfrac1n, E_1+\tfrac1n])) \cap
  \textup{Ran}(\Pi_Q([E_2-\tfrac1n, E_2+\tfrac1n])).
  \]
  Then
  \begin{align*}
    \abs{\pscal{(P-E_1)^2u_n}{u_n}} = & \abs{\int_{[E_1-\tfrac1n,
        E_1+\tfrac1n]} (\lambda-E_1)^2 \pscal{\Pi_P({\rm d}\lambda)u_n}{u_n}}\\
    \leq & \frac{1}{n^2} \int \abs{\pscal{\Pi_P({\rm
          d}\lambda)u_n}{u_n}} \leq \frac{1}{n^2}.
  \end{align*}
  Similarly, $\abs{\pscal{(Q-E_2)^2u_n}{u_n}}\leq
  \frac{1}{n^2}$. Letting $n\to\infty$, we
  contradict~\eqref{eq:positive}. Thus $E\not\in
  \textup{JointSpec(P,Q)}$, which proves the proposition.

\end{proof}

\subsection{Bohr\--Sommerfeld rules}
Recall that the \emph{Hausdorff distance} between two bounded subsets
$A$ and $B$ of $\R^2$ is
 $$
 {\rm d}_H(A,\,B):= \inf\{\epsilon > 0\,\, | \,\ A \subseteq
 B_\epsilon \ \mbox{and}\ B \subseteq A_\epsilon\},
$$
where for any subset $X$ of $\R^2$, the set $X_{\epsilon}$ is
$$X_\epsilon := \bigcup_{x \in X} \{m \in \R^2\, \, | \,\, \|x - m \|
\leq \epsilon\}.$$ If $(A_{\hbar})_{\hbar \in I}$ and
$(B_{\hbar})_{\hbar \in I}$ are sequences of uniformly bounded subsets
of $\mathbb{R}^2$, we say that $$A_{\hbar} = B_{\hbar} +
\mathcal{O}(\hbar^{N}) $$ if there exists a constant $C>0$ such that
$$
{\rm d}_H(A_{\hbar},\,B_{\hbar})\leq C\hbar^{N}$$ for all $\hbar \in
I$.  If $A$ or $B$ are not uniformly bounded, we shall say that
$A_{\hbar} = B_{\hbar} + \mathcal{O}(\hbar^{N}) $ \emph{on a ball $D$}
if there exists a sequence of sets $D_\h$, all diffeomorphic to $D$,
such that $D_\h = D + \mathcal{O}(\h^2)$ and
\[
A_\h \cap D = B_\h \cap D_\h + \mathcal{O}(\h^2).
\]

\begin{definition} \label{chart} Let $F:=(J,H) \colon M \to
  \mathbb{R}^2$ be an integrable system on a $4$\--dimensional
  connected symplectic manifold. Let $P$ and $Q$ be commuting
  semiclassical operators with principal symbols $J,H \colon M \to
  \mathbb{R}$. We say
  that $\op{JointSpec}(P,\,Q)$ \emph{satisfies the Bohr\--Sommerfeld
    rules} 
  if for every regular value $c$ of $F$ there exists a small ball
  ${\rm B}(c,\epsilon_c)$ centered at $c$, such that,
  \begin{equation} \op{JointSpec}(P,\,Q) = g_{\hbar}(2\pi \hbar
    \Z^2\cap D) + \mathcal{O}(\hbar^2) \quad \text{ on } {\rm
      B}(c,\epsilon_c),
    \label{eq:bs1} \end{equation}
  with $$g_{\hbar}=g_0+\hbar g_1,$$ where $g_0,g_1$ are smooth maps
  defined on a bounded open set $D\subset\R^2$, $g_0$ is a
  diffeomorphism into its image, $c \in g_0(D)$ and the components of
  $g_0^{-1}=(\mathcal{A}_1,\,\mathcal{A}_2)$ form a basis of action
  variables.
\end{definition}

\smallskip In this situation, if $\hbar$ is small enough, then
$g_{\hbar}$ is a diffeomorphism into its image, and its inverse admits an asymptotic expansion in non-negative powers of $\hbar$ for the $\mathcal{C}^{\infty}$ topology; we call
$(g_{\hbar})^{-1}$ an \emph{affine chart} for $\op{JointSpec}(P,\,Q)$.

Bohr\--Sommerfeld rules are known to hold for integrable systems of
pseudodifferential operators (thus $M$ is a cotangent bundle)
\cite{charbonnel,vungoc-focus}, or for integrable systems of Toeplitz
operators on prequantizable compact symplectic
manifolds~\cite{charles-quasimodes}. It would be interesting to
formalize the minimal semiclassical category where Bohr-Sommerfeld
rules are valid.

Note that action variables are not unique. Thus, if $(g_{\hbar})^{-1}$
is an affine chart for ${\rm JointSpec}(P,Q)$ and $B \in {\rm
  GL}(2,\Z)$ then $B \circ (g_{\hbar})^{-1}$ is again an affine chart.
In view of the discussion in Section~\ref{subsect:period}, this remark
implies the following proposition.
\begin{prop}
  If $F$ is a semitoric system, then in Definition~\ref{chart}, we can
  assume that $\mathcal{A}_1(c_1,c_2)=c_1$. Therefore, there exists an
  integer $k$ such that the actions $\mathcal{A}_1, \mathcal{A}_2$
  satisfy:
  \begin{equation} {\rm d}\mathcal{A}_1 = {\rm d}c_1, \quad {\rm
      d}\mathcal{A}_2 = (\zeta_1 + k) {\rm d}c_1 + \zeta_2 {\rm
      d}c_2, \label{eq:actions}
  \end{equation}
  where $\zeta_1,\zeta_2$ are defined in~\eqref{equ:zeta}.
\end{prop}

\section{Main Result}
We state in this section a more precise version of our main result,
Theorem~\ref{thm:main}.  Let $\mathcal{M}_{\textup{ST}}$ be the set of
semitoric systems (\emph{i.e.} triples $(M,\omega,F)$ satisfying
Definition~\ref{def:semitoric}) modulo isomorphisms (as defined in
Definition~\ref{def:iso}).

For each of the four invariants (1), (2), (3), or (4) mentioned in
Section~\ref{sec:introduction}, we may define a map $\mathcal{I}_j$,
$j=1,2,3,4$, from $\mathcal{M}_{\textup{ST}}$ with value in the
appropriate space corresponding to the invariant (we refer
to~\cite{PeVN2009} for these spaces; here we simply denote them by
$\mathcal{B}_j$, $j=1,2,3,4$, as their precise definition is not
important for our purpose).

Let $\mathcal{Q}_{\textup{ST}}$ be the set of all quantum \emph{simple} semitoric
systems \emph{for which the Bohr\--Sommerfeld rules hold}, equipped with the natural arrow 
$$\sigma:
\mathcal{Q}_{\textup{ST}}\to\mathcal{M}_{\textup{ST}}$$ induced by the
principal symbol map. We introduce now the \emph{joint spectrum map}
\begin{align*}
  \textup{JS} : \mathcal{Q}_{\textup{ST}} & \longrightarrow \mathcal{P}(\R^2)^I\\
  (P,Q) & \longmapsto \textup{JointSpec}(P,Q),
\end{align*}
where we recall that $I$ is the set where the semiclassical parameter
$\h$ varies. Let us denote by $\mathcal{P}_2$ the set of equivalence
classes of $\h$-dependent subsets of $\R^2$ with respect to the
equality modulo $\mathcal{O}(\h^2)$ on every ball, and
$\overline{\textup{JS}}: \mathcal{Q}_{\textup{ST}} \to \mathcal{P}_2$
the quotient map of $\textup{JS}$. Let $\Sigma\subset\mathcal{P}_2$ be
the range of $\overline{\textup{JS}}$, \emph{i.e.} the subset of all
joint spectra of semitoric systems, modulo $\mathcal{O}(\h^2)$. Then
Theorem~\ref{thm:main} can be rephrased as follows:
\begin{theorem}
  \label{thm:main-precise}
  For each $j=1,2,3,4$, there exists a map $\hat{\mathcal{I}}_j :
  \Sigma \to \mathcal{B}_j$ such that the following diagram is
  commutative:
  \[
  \xymatrix{ \mathcal{Q}_{\textup{ST}} \ar[r]^{\overline{\textup{JS}}}
    \ar[d]^{\sigma} &
    \Sigma \ar@{=>}[d]^{\hat{\mathcal{I}}_j}\\
    \mathcal{M}_{\textup{ST}} \ar[r]^{\mathcal{I}_j} & \mathcal{B}_j }
  \]
\end{theorem}

\begin{cor}
  If two quantum simple semitoric systems for which the Bohr\--Sommerfeld rules hold 
  have the same joint spectrum
  modulo $\mathcal{O}(\h^2)$, then the underlying classical systems
  have the same set of invariants {\rm (1)}, {\rm (2)}, {\rm (3)},
  {\rm (4)}. In particular, if two quantum Jaynes\--Cummings type systems 
  for which the Bohr\--Sommerfeld rules hold have the same joint spectrum
  modulo $\mathcal{O}(\h^2)$, then the underlying classical systems are isomorphic.
\end{cor}

\section{Proof of Theorem~\ref{thm:main-precise}}

Let $P,Q$ be a quantum simple semitoric system with joint principal symbol $F=(J,H)$. Remember that we want to prove that the knowledge of the joint spectrum of $P,Q$ modulo $\mathcal{O}(\hbar^2)$ allows to recover invariants $(1)$ to $(4)$. For the sake of clarity, we divide the proof into five steps.
\\
\paragraph{\bf Step 1} We recover the image $F(M)$ thanks to
Proposition~\ref{prop:joint-spectrum-image}. Indeed, choose a point $E
= (E_1,E_2)$ in $\R^2$; assume that the following condition holds:
\begin{itemize}
\item[(C)] for every $\varepsilon > 0$ and for every $\hbar_0 \in I$,
  there exists $\hbar \leq \hbar_0$ in $I$ such that
  $\text{JointSpec}(P_{\hbar},Q_{\hbar}) \cap B(E,\varepsilon) \neq
  \emptyset$.
\end{itemize}
Then Proposition~\ref{prop:joint-spectrum-image} implies that $E$
belongs to $F(M)$.  Conversely, assume $E\in B_r$, where $B_r$ is the
set of regular values of $F$. Because of the Bohr-Sommerfeld rules,
there exists a small ball around $E$ in $\R^2$ in which the joint
spectrum is a deformation of the lattice $\h\Z^2$. Hence when $\h$ is
small enough, this ball always contains some element of the joint
spectrum (the number of joint eigenvalues grows like $\h^{-2}$), which
says that Condition~(C) holds. 
Let $S$ be the set of $E\in\R^2$ for which (C) holds. We have
\[
B_r \subset S \subset F(M).
\]
But we know from~\cite[Proposition $2.9$]{VN2007} that the closure of
$B_r$ equals $F(M)$. Therefore, $\overline{S}=F(M)$, which proves that
the image $F(M)$ can be recovered from the joint spectrum.

Note that this step would also work with a weaker hypothesis than the
Bohr-Sommerfeld rules. For instance, having a $\mathcal{C}^\infty_0$
functional calculus for the semiclassical operators, or being able to
construct microlocal quasimodes (which is common in pseudodifferential
or Toeplitz analysis) would be sufficient for recovering $F(M)$.
\\
\paragraph{\bf Step 2}
In this step, we show how to recover the periods of the classical
system at regular values from the knowledge of the joint spectrum. In
order to do so, we adapt an argument from~\cite{VN2011} for the
resolution of a similar inverse problem in dimension $2$. Although in
our case we are working in dimension $4$, which makes the study more
difficult, the situation is also simpler by some aspects, because we
know from \cite[Theorem $3.4$]{VN2007} that the regular fibers are
connected.
\ \\

Let $c_0$ be a regular value of $F$, and let $B$ be a ball centered at
$c_0$ in which the joint spectrum is described by the Bohr-Sommerfeld
rules (\ref{eq:bs1}). Let $D$ and $g_{\hbar}$ be as in the statement
of the latter. We can assume that $g_{\hbar}$ is a diffeomorphism from
$g_{\hbar}^{-1}(B)$ into $B$. We recall that
\[ \op{JointSpec}(P_{\hbar},\,Q_{\hbar}) \cap B = g_{\hbar}(2\pi \hbar
\Z^2 \cap D) \cap B_\h + \mathcal{O}(\hbar^2), \] where $B=B_\h +
\mathcal{O}(\h^2)$.
\ \\
Now, let $\chi$ be a non-negative smooth function with compact support
$K \subset B$, equal to $1$ on a compact subset of $B$. We consider
the spectral measure
\[ D(\lambda,\hbar) = \sum_{c \in \mathrm{JointSpec}(P,Q) \cap B}
\chi(c) \delta_c(\lambda) \] where $\delta_c$ is the Dirac
distribution at $c$. Let $\mathcal{F}_{\hbar}$ stand for the
semiclassical Fourier transform, so that
\[ \mathcal{F}_{\hbar}(f)(\xi) = \frac{1}{(2\pi \hbar)^2} \int_{\R^2}
\exp\left( -i \hbar^{-1} \langle x, \xi \rangle \right) f(x) {\rm
  d}x, \] for smooth, compactly supported functions $f$, and introduce
\[ Z(t,\hbar) = (2\pi \hbar)^2 \mathcal{F}_{\hbar}(D(\cdot,\hbar))(t)
= \sum_{c \in \mathrm{JointSpec}(P,Q) \cap B} \chi(c) \exp\left( -i
  \hbar^{-1} \langle c, t \rangle \right). \] Thanks to the
Bohr-Sommerfeld conditions, we may estimate this quantity as
\[ Z(t,\hbar) = \sum_{s\in 2\pi\hbar\Z^2 \cap D}
\varphi_t(g_{\hbar}(s),\hbar) + \mathcal{O}(\hbar) \] with
\[ \varphi_t(s,\hbar) = \chi(g_{\hbar}(s)) \exp\left( -i \hbar^{-1}
  \langle g_{\hbar}(s), t \rangle \right).  \] Because
$\chi(g_{\hbar}(s)) = 0$ if $s \notin D$, this yields
\[ Z(t,\hbar) = \sum_{\alpha \in \Z^2} \varphi_t(g_{\hbar}(2\pi \hbar
\alpha),\hbar) + \mathcal{O}(\hbar). \] By the Poisson summation
formula, we thus obtain
\[ Z(t,\hbar) = \sum_{\beta \in \Z^2} Z_{\beta}(t,\hbar) +
\mathcal{O}(\hbar) \] with
\[ Z_{\beta}(t,\hbar) = \frac{1}{(2\pi \hbar)^2} \int_{\R^2}
\exp\left( -i \hbar^{-1} \left( \langle \beta, s \rangle + \langle
    g_{\hbar}(s), t \rangle \right) \right) \chi(g_{\hbar}(s)) {\rm
  d}s. \] Since $g_{\hbar}$ is a diffeomorphism from
$g_{\hbar}^{-1}(B)$ into $B$, we can use the change of variables $c =
g_{\hbar}(s)$, $s = f_{\hbar}(c)$, which yields:
\[ Z_{\beta}(t,\hbar) = \frac{1}{(2\pi \hbar)^2} \int_{\R^2}
\exp\left( -i \hbar^{-1} \left( \langle \beta, f_{\hbar}(c) \rangle +
    \langle c, t \rangle \right)\right) \chi(c) |\det
J_{f_{\hbar}}(c)| {\rm d}c, \] which means that $Z_{\beta}(t,\hbar) =
\mathcal{F}_{\hbar}(\psi_{\beta})(t)$ where
\[ \psi_{\beta}(c) = \exp\left( -i\hbar^{-1} \langle
  \beta,f_{\hbar}(c) \rangle \right) \chi(c) |\det
J_{f_{\hbar}}(c)| \] is a WKB function with phase
\[ \theta_{\beta}(c) = - \langle \beta,f_0(c) \rangle = - \beta_1
\mathcal{A}_1(c) - \beta_2 \mathcal{A}_2(c). \] Since by equation
(\ref{eq:actions})
\[ \nabla \theta_{\beta}(c) = -\begin{pmatrix} \beta_1 + \beta_2(
  \zeta_1(c) + k) \\ \beta_2 \zeta_2(c) \end{pmatrix}, \] the
associated Lagrangian submanifold is the set
\begin{equation} \left\{ (c,t) \in \R^4 \, |  \quad (t_1,t_2) = - (\beta_1
    + \beta_2( \zeta_1(c) + k), \beta_2 \zeta_2(c))
  \right\}. \label{eq:lag}\end{equation} One can easily check that
this submanifold is indeed Lagragian: the $1$-form $\nu = \left(
  \beta_1 + \beta_2( \zeta_1(c) + k ) \right) {\rm d}c_1 + \beta_2
\zeta_2(c) {\rm d}c_2$ is closed, as
\[ \nu = d(\beta_1 \mathcal{A}_1 + \beta_2 \mathcal{A}_2). \]
Since the Jacobian $|\det J_{f_{\hbar}}(c)|$ does not vanish in the
support $K$ of $\chi$, this implies that the semiclassical wavefront
set of $Z_{\beta}(\cdot,\hbar)$ is
\begin{align*}
  \mathrm{WF}_{\hbar}(Z_{\beta}(\cdot,\hbar)) = & \left\{ (c,t) \in
    \R^4 \, | ~ (t_1,t_2) = - (\beta_1 + \beta_2( \zeta_1(c) + k), \beta_2
    \zeta_2(c)), \ c\in K \right\} \\ = & \, \mathcal{L}_{\beta}(K)
\end{align*}

\ \\
To obtain a similar result on $Z(\cdot,\hbar)$, we still need to sum
over $\beta \in \Z^2$. Let $t^0 = (t_1^0, t_2^0) \in (\R^*)^2$,
$\varepsilon > 0$, and let $\rho \in
\mathcal{C}^{\infty}_0(B(t_0,\varepsilon))$.
\begin{lemma}
  If there exists a solution $(c,t)$ of \eqref{eq:lag} with $t$ in the
  support of $\rho$, then $\beta$ is such that
  \[ \max(|\beta_1|,|\beta_2|) \leq M \] where $M$ is defined as
  \[ M = \frac{\varepsilon + \|t^0\|}{\min_K |\zeta_2|} \max \left(1,
    \min_K |\zeta_2| + |k| + \max_{K} |\zeta_{1}| \right). \]
\end{lemma}
\begin{proof}
  For such a solution, we have $\|t\| \leq \varepsilon + \|t^0\|$,
  thus
  \[ \left( \beta_1 + \beta_2( \zeta_1(c) + k) \right)^2 + \beta_2^2
  \zeta_2(c)^2 \leq (\varepsilon + \|t^0\|)^2, \] which implies that
  \begin{equation} |\beta_2| \leq \frac{\varepsilon + \|t^0\|}{\min_K
      |\zeta_2|}. \label{eq:maj1}\end{equation} Since we also have
  \[ | \beta_1 + \beta_2( \zeta_1(c) + k) | \leq \varepsilon + \|t^0\|
  , \] we deduce from the previous inequality that
  \begin{equation} | \beta_1 | \leq (\varepsilon + \|t^0\|) \left( 1 +
      \frac{|k| + \max_K |\zeta_1|}{\min_K |\zeta_2|}
    \right), \label{eq:maj2}\end{equation} which proves the result.
\end{proof}
Using the proof of the non-stationary phase lemma, we can write for
such a $\beta$ and any $N \geq 1$
\[ (2\pi \hbar)^2 Z_{\beta}(t,\hbar)\] equals
\[ \left(\frac{i\hbar}{\max(|\beta_1|,|\beta_2|)}\right)^N \int_{\R^2}
\exp\left( -i\hbar^{-1} (\langle \beta, f_{\hbar}(c) \rangle + \langle
  c, t \rangle) \right) L^N(a(c,\hbar)) {\rm d}c, \] where
$a(\cdot,\hbar)$ is compactly supported and admits an asymptotic
expansion in non-negative powers of $\hbar$ in the
$\mathcal{C}^{\infty}$ topology, and $L$ is the differential operator
defined as
\[ Lu = \nabla\left( \max(|\beta_1|,|\beta_2|) \frac{u}{|V|^2} V
\right) \] with
\[ V(c) = - \begin{pmatrix} t_1 + \beta_1 + \beta_2( \zeta_1(c) + k)
  \\ t_2 + \beta_2 \zeta_2(c) \end{pmatrix}. \] Introduce the function
$b = (\max(|\beta_1|,|\beta_2|)/|V|^2)V$; one has
\[ |b(c)| = \left( \left( \frac{t_1 + \beta_1 + \beta_2( \zeta_1(c) +
      k)}{\max(|\beta_1|,|\beta_2|)} \right)^2 + \left( \frac{t_2 +
      \beta_2 \zeta_2(c)}{\max(|\beta_1|,|\beta_2| )} \right)^2
\right)^{-1/2}. \] Then $b$ is uniformly bounded on $K$ for $\beta$
such that $\max(|\beta_1|,|\beta_2|) > M$ and, for every $\ell \in
\mathbb{N}^2$, there exists a constant $C_{\ell}$ such that
\[ |\partial_{c^\ell} b| =
|\partial_{{c_1}^{\ell_1}} \partial_{{c_2}^{\ell_2}} b| \leq
C_{\ell} \]
on $K$. Consequently, there exists a constant $\tilde{C}_N > 0$ such
that
\[ |\rho(t) Z_{\beta}(t,\hbar)| \leq \tilde{C}_N
\left(\frac{\hbar}{\max(|\beta_1|,|\beta_2|)}\right)^N \] when
$\max(|\beta_1|,|\beta_2|) > M$. Therefore, for $N \geq 4$, we have
\[ \sum_{\substack{\beta \in \Z^2 \\ \max(|\beta_1|,|\beta_2|) > M }}
|\rho(t) Z_{\beta}(t,\hbar)| \leq \hat{C}_N \hbar^N \] for some
constant $\hat{C}_N > 0$. This shows that only a finite number of
terms contribute to $\rho(t) Z(t,\hbar)$ up to
$\mathcal{O}(\hbar^{\infty})$, hence
\begin{align*}
  \mathrm{WF}_{\hbar}(\rho Z(\cdot,\hbar)) \subset & \left\{ (c_1,c_2,-
    \beta_1 - \beta_2( \zeta_1(c) + k) , -\beta_2 \zeta_2(c)) \in \R^4
    \right. |\\ ~ &  \,\,\, \left.(c_1,c_2) \in K, \max(|\beta_1|,|\beta_2|) > M
  \right\} 
\end{align*}
and finally
\begin{align*}
  \mathrm{WF}_{\hbar}(Z(\cdot,\hbar)) = & \left\{ (c_1,c_2,- \beta_1 -
    \beta_2( \zeta_1(c) + k) , -\beta_2 \zeta_2(c)) \in \R^4 \right. | \\
  ~&  \,\,\,  \left.(c_1,c_2) \in K, \beta \in \Z^2 \right\}\\
  = & \, \mathcal{L}(K),
\end{align*}
which is exactly the restriction of the period bundle over $K$ (see Section \ref{subsect:period}).\\
\ \\
The last part of this step is to explain how one can extract the
functions $(\tau_1,\tau_2)$ from the data of $\mathcal{L}(K) =
\bigcup_{\beta \in \Z^2}\mathcal{L}_{\beta}(K)$, which is the disjoint
union of smooth surfaces in $\R^4$. Endow $\R^4$ with the coordinates
$(x_1,x_2,x_3,x_4)$, and introduce the plane $\Pi = \{ x \in \R^4\, | \
x_1 = c^0_1, \ x_2 = c^0_2 \}$, for a fixed $c^0 \in K$.  Then the set
\[ \mathcal{E} = \mathcal{L}(K) \cap \Pi = \{ (c_1^0,c_2^0,- \beta_1 -
\beta_2( \zeta_1(c^0) + k) , -\beta_2 \zeta_2(c^0))\, | \ \beta \in \Z^2
\} \] is discrete, and the set $\{ x_4 \, |  \ x \in \mathcal{E}\} \cap
\R^{*}_{+}$ is bounded from below. Let $\mathcal{F}$ be the set of
points in $\mathcal{E}$ with minimal coordinate $x_4$; then
\[ 
\mathcal{F} = \{ (c_1^0,c_2^0,\zeta_1(c^0) + k -
\beta_1,\zeta_2(c^0))\, | \ \beta_1 \in \Z \}. 
\] Again, the set $\{ x_3
\, |\,\, x \in \mathcal{F}\} \cap \R^{*}_{+}$ is bounded from below,
and the point of this set with minimal coordinate $x_{3}$ is
$(c_1^0,c_2^0,\zeta_1(c^0),\zeta_2(c^0))$. The connected component of
this point in $\mathcal{L}(K)$ is the graph of the function 
\[ c \in K \mapsto (\zeta_{1}(c),\zeta_{2}(c)) = \frac{1}{2\pi} (\tau_1(c),\tau_2(c)). \]

\ \\
\paragraph{\bf Step 3} 
Let us now explain how to recover the position of the focus-focus
values from the joint spectrum. Thanks to step $1$, we know $F(M)$. By
\cite[Theorem $3.4$]{VN2007}, we know that the boundary of $F(M)$
consists of the singularities of elliptic-elliptic and transversally
elliptic type, and that the only singular values in the interior of
$F(M)$ are the images of the focus-focus singularities.  Let $A$ be any
point lying on the boundary $\partial F(M)$. Let $C_{1}, \ldots,
C_{m_{f}}$ be the images of the focus-focus points in $F(M)$, labelled
in such a way that
\[ J(m_{1}) < J(m_{2}) < \ldots < J(m_{m_{f}}),\] where for $i$ in $\{
1, \ldots, m_{f} \}$, $m_{i}$ is the only focus-focus-point in
$F^{-1}(C_{i})$.  Consider the distance $d = \min_{1 \leq i \leq m_f} \| A - C_{i} \|$ and let $j \in \{1,\ldots,m_f\}$ be such that $d = \| A - C_{j} \|$; since $C_{j}$ lies in the interior of $F(M)$, we have that $d >
0$. Let $B_r$ be the set of regular values of $F$; for every
$\varepsilon$ in $(0,d]$, the intersection
\[ X_{\varepsilon} = B(A,\varepsilon) \cap \mathring{F(M)} \] of the
ball of radius $\varepsilon$ centered at $A$ with the interior of
$F(M)$ is contained in $B_r$. Thus, from step $2$, we can compute the
function ${\tau_2}_{|X_{\varepsilon}}$ from the joint spectrum. It
follows from \cite[proposition $3.1$]{VN2003} that $\tau_2$ has a
logarithmic behavior near $C_{j}$. Hence, if
${\tau_2}_{|X_{\varepsilon}}$ can be extended to a continuous function
on $\bar{B}(A,\varepsilon) \cap \mathring{F(M)}$, then necessarily
$\varepsilon < d$. This allows to find $d$; the point $C_{j}$ belongs
to the circle $\mathcal{C}$ of radius $d$ centered at
$A$. Furthermore, the only points in $\mathcal{C} \cap \mathring{F(M)}$
where $\tau_{2}$ admits a logarithmic singularity are some of the $C_i$ (including $C_j$), that we
recover this way.

We obtain the positions of the other focus-focus values by applying
this method recursively. For instance, we recover another point $C_{k}$ by
considering circles of growing radius centered at $C_{j}$, and so on (let us recall that $m_f$ is finite).
\ \\
\paragraph{\bf Step 4}
Since we now know precisely the position of the focus-focus values,
\cite[Theorem~3.3]{PeVN2013} implies that the Taylor series invariant
associated with each focus-focus singularity can be recovered from the
joint spectrum.

\ \\
\paragraph{\bf Step 5}
In this step, we prove that from the data of the joint spectrum, one
can deduce the polygonal invariant introduced in \cite{VN2007}.

Recall that a map $U \subset \R^n \to V \subset \R^n$ is
\emph{integral affine} on $U$ if it is of the form $x \in U \mapsto Ax
+ b$, where $A \in \text{GL}(n,\Z)$ and $b \in \R^n$. An
\emph{integral affine structure} on a smooth $n$-dimensional manifold
is the data of an atlas $(U_{i},\varphi_{i})$ such that for all $i$,
the transition function $\varphi_{i} \circ \varphi_{j}^{-1}$ is
integral affine. 
 
As a consequence of the action-angle theorem, the integrable system
$(J,H)$ induces an integral affine structure on the set $B_{r}$ of
regular values of $F$. The charts are action variables, that is maps
$\varphi: U \to \R^2$ where $U$ is a small open subset of $B_{r}$ and
$\varphi \circ F$ generates a $\mathbb{T}^2$-action.

Let $\varepsilon_{i} \in \{-1,1\}$ and let
$\ell_{i}^{\varepsilon_{i}}$ be the vertical segment starting at the
focus-focus value $C_{i}$, going upwards (respectively downwards) if
$\varepsilon_{i} = 1$ (respectively $\varepsilon_{i} = -1$), and
ending at the boundary of $F(M)$. Set $ \ell^{\vec{\varepsilon}} =
\bigcup_{i} \ell_{i}^{\varepsilon_{i}}$.

\begin{theorem}[{\cite[Theorem $3.8$]{VN2007}}]
\label{theo:delta}
  For $\vec{\varepsilon} \in \{-1,1\}^{m_{f}}$, there exists a
  homeomorphism $\Phi$ from $B=F(M)$ to $\Delta = \Phi(B) \subset
  \R^2$ such that:
  \begin{enumerate}
  \item $\Phi_{|B\setminus \ell^{\vec{\varepsilon}}}$ is a
    diffeomorphism into its image,
  \item $\Phi_{|B_r\setminus \ell^{\vec{\varepsilon}}}$ is affine: it
    sends the integral affine structure of $B_{r}$ to the standard
    integral affine structure of $\R^2$,
  \item $\Phi$ preserves $J$: $\Phi(x,y) = (x,\Phi_{2}(x,y))$,
  \item $\Phi_{|B_r\setminus \ell^{\vec{\varepsilon}}}$ extends to a
    smooth multi-valued map from $B_{r}$ to $\R^2$ and for any $i \in
    \{1,\ldots,m_{f}\}$ and any $c \in \mathring{\ell}_{i}$, then
    \[ \lim_{\substack{(x,y) \to c \\ x < x_{i}}} {\rm d} \Phi(x,y)
    = \begin{pmatrix} 1 & 0 \\ 1 & 1 \end{pmatrix}
    \lim_{\substack{(x,y) \to c \\ x > x_{i}}} {\rm d} \Phi(x,y), \]
  \item $\Delta$ is a rational convex polygon.
  \end{enumerate}
\end{theorem}

The polygon $\Delta$ is the sought invariant; in fact, the real
invariant is a family of such polygons, more precisely the set of all
such $\Delta$ for all possible choices of $\vec{\varepsilon} \in
\{-1,1\}^{m_{f}}$ and all their images by linear maps leaving the
vertical direction invariant. We refer the reader to \cite[Section
$4.3$]{PeVN2009} for more precise statements.

\begin{prop}\label{prop:affine}
  Given any $\vec{\varepsilon} \in \{-1,1\}^{m_{f}}$, the
  corresponding polygon $\Delta=\Delta_{\vec\varepsilon}$ is
  determined by the integral affine structure of $B_r$.
\end{prop}
\begin{proof}
  Once a starting point $c_0\in B_r$ is chosen (which, by convention,
  is taken to be on the left of the first focus-focus critical value,
  when these values are ordered by non-decreasing abscissae), the
  affine map $\Phi_{|B_r\setminus \ell^{\vec{\varepsilon}}}$ is
  uniquely determined by the affine structure. Indeed, the set
  $B_r\setminus \ell^{\vec{\varepsilon}}$ is simply connected and
  $\Phi$ is the developing map of the induced affine structure. The
  crucial observation is that it follows from the construction
  in~\cite{VN2007} that the map $\Phi_{|B\setminus
    \ell^{\vec{\varepsilon}}}$ is the natural extension of
  $\Phi_{|B_r\setminus \ell^{\vec{\varepsilon}}}$ to the boundary of
  $B_r\setminus \ell^{\vec{\varepsilon}}$ away from the half-lines
  $\ell^{\vec{\varepsilon}}$, and this boundary consists of elliptic
  (or transversally elliptic) singularities. Precisely, the extension
  is obtained as follows. Near a 1-dimensional family of transversally
  elliptic singularities, we use the normal form due to Miranda and
  Zung \cite{MirZung}: if $c_{e}$ is a transversally elliptic value,
  there exist a symplectomorphism $\varphi$ from a neighborhood of
  $F^{-1}(c_{e})$ in $M$ to a neighborhood of $\{ I = x = \xi = 0 \}$
  in ${\rm T}^*S^1 \times \R^2$ with coordinates
  $((\theta,I),(x,\xi))$ and standard symplectic form ${\rm d}I \wedge
  {\rm d} \theta + {\rm d}\xi \wedge {\rm d} x$, which sends the set
  $\{ F=\text{constant} \}$ to the set $\{I=\text{constant},
  x^2+\xi^2= \text{constant}\}$, and a smooth function $g$ such that
  \[ (F \circ \varphi^{-1})(\theta,I,x,\xi) = g(I,x^2+\xi^2) \] where
  $\varphi^{-1},g$ are defined. Let $(\mathcal{A}_1,\mathcal{A}_2)$ be
  an affine chart for $B_r$ inside this neighborhood where the normal
  form holds. Since $(I,(x^2+\xi^2)/2)$ is also an affine chart, there
  exists a matrix $A = \begin{pmatrix} \alpha & \beta \\ \gamma &
    \delta \end{pmatrix} \in \text{GL}(2,\Z)$ such that
  \begin{equation}
  \mathcal{A}_{1} (c) = \alpha I + \beta (x^2 + \xi^2)/2; \quad
  \mathcal{A}_{2} (c) = \gamma I+ \delta (x^2 + \xi^2)/2\label{equ:actions2}
\end{equation}
where $m = \varphi^{-1}(\theta,I,x,\xi) \in M$ is a regular point for
$F$ and $c = F (m)$. Since $\alpha,\beta,\gamma, \delta$ are constant,
Formula~\eqref{equ:actions2} naturally gives the required extension of
$(\mathcal{A}_1,\mathcal{A}_2)$ (and hence $\Phi$) to the boundary
near $c_{e}$. Near an elliptic-elliptic point, we can apply the same
reasoning, using Eliasson's normal form \cite{Elia}.
\end{proof}

In view of the proposition, step 5 will be treated as soon as we show
that the integral affine structure on $B_r$ can be recovered from the
joint spectrum ${\rm JointSpec}(P,Q)$ up to $\mathcal{O}(\hbar^2)$,
which can be done as follows.  From the previous steps, we can recover
$F(M)$ and the position of the focus-focus values
$C_{i}=(x_{i},y_{i})$, $1 \leq i \leq m_{f}$. Therefore, we know the
set of regular values $B_r$, which is the interior of $F(M)$ minus the
focus-focus critical values.

In a small ball $B_0\subset B_r$, we can construct action variables
$\left(\mathcal{A}_{1}, \mathcal{A}_{2}\right)$. Indeed, from step 2
we can recover the functions $\tau_{1}, \tau_{2}$ on $B_{0}$. Fixing a point $s \in B_{0}$, we can pick for every point $c
\in B_{0}$ a smooth path $\gamma_{c}: [0,1] \to B_{0}$ such that
$\gamma_{c}(0) = s$, $\gamma_{c}(1)=c$ and compute
\[ \mathcal{A}_{1}^{(0)}(c) = c_{1}; \quad \mathcal{A}_{2}^{(0)}(c) =
\int_{0}^1 \left\langle \begin{pmatrix} \zeta_{1}(\gamma_{c}(t)) \\
    \zeta_{2}(\gamma_{c}(t)) \end{pmatrix}, \gamma_{c}'(t)
\right\rangle {\rm d}t, \] where we recall that $\tau_{i} =
2\pi\zeta_{i}$ for $i=1,2$. In this way, we have constructed the
integral affine structure of $B_r$ from the joint spectrum. In remains
to apply Proposition~\ref{prop:affine} to construct $\Phi$, and hence
$\Delta$ by Theorem~\ref{theo:delta}.

~\\
\paragraph{\bf Step 6}
It only remains to prove that we can recover the height invariant
associated with each focus-focus singularity from the joint
spectrum. In order to do so, let $i \in \{1,\ldots,m_{f}\}$ and
consider a sequence $(Y_{n})_{n \in \N}$ of points of $B$ such that
every $Y_{n}$ has the same abscissa as $C_{i}$ and ordinate smaller
than the one of $C_{i}$, and such that $Y_n \underset{n \to + \infty}{\longrightarrow} C_i$. We may assume that $Y_{0}$ lies on the
boundary of $B$. Let $\Phi$ be a homeomorphism from $B$ to $\Delta$ as
in the previous step. Then the point
\[ P = \lim_{n \to +\infty} \Phi(Y_{n}) \] is well-defined and $P$
is the image of the focus-focus value in the polygon $\Delta$. The
height invariant that we seek is the difference between the ordinate
of $P$ and the ordinate of $\Phi(Y_{0})$.

\begin{remark}
  Another way of obtaining the height invariant associated with
  $C_{i}$ would have been to count the joint eigenvalues lying on a vertical
  line below $C_{i}$ and use a Weyl law to relate this number to the
  volume of the set $J^{-1}(C_{i}) \cap \{ H < H(m_{i}) \}$. Although it may
  seem more natural than our method, it is also more technical, and
  that is why we have chosen not to treat the problem this way.

\end{remark}

\vskip 1em

{\emph{Acknowledgements}.  Part of this paper was written at the
  Institute for Advanced Study (Princeton, NJ) during the visit of the
  last two authors in July 2014, and they are very grateful to Helmut
  Hofer for the hospitality. AP was partially supported by  NSF DMS-1055897, the STAMP Program at the ICMAT research institute
  (Madrid), and ICMAT Severo Ochoa grant Sev-2011-0087. VNS is partially supported by the Institut Universitaire de
  France, the Lebesgue Center (ANR Labex LEBESGUE), and the ANR
  NOSEVOL grant.}

\noindent

\medskip\noindent

\noindent
\noindent
{\bf Yohann Le Floch} \\
Institut de Recherches Math\'ematiques de Rennes\\
Universit\'e de Rennes 1\\
Campus de Beaulieu\\
F-35042 Rennes cedex, France\\\\
{\em E-mail:} \texttt{yohann.lefloch@univ-rennes1.fr}\\
{\em Website}: \url{http://perso.univ-rennes1.fr/yohann.lefloch/}

\medskip\noindent

\noindent
\\
{\bf {{\'A}lvaro Pelayo}}\\
Department of Mathematics\\
University of California, San Diego\\
9500 Gilman Drive  \# 0112\\
La Jolla, CA 92093-0112, USA\\
{\em E\--mail}: \texttt{pelayo.alv@gmail.com}

\medskip\noindent

\noindent
\noindent
{\bf San V\~u Ng\d oc} \\
Institut Universitaire de France
\\
\\
Institut de Recherches Math\'ematiques de Rennes\\
Universit\'e de Rennes 1\\
Campus de Beaulieu\\
F-35042 Rennes cedex, France\\\\
{\em E-mail:} \texttt{san.vu-ngoc@univ-rennes1.fr}\\
{\em Website}: \url{http://blogperso.univ-rennes1.fr/san.vu-ngoc/}

\end{document}